\documentclass[english]{article}
\usepackage[T1]{fontenc} 
\usepackage{textcomp,lmodern} 
\usepackage[utf8]{inputenc} 
\usepackage[babel=true]{microtype} 
\usepackage{babel} 
\usepackage{mathtools} 
\usepackage{amssymb} 
\usepackage{amsfonts} 
\usepackage{amsthm} 
\usepackage{xcolor} 
\usepackage{graphicx} 
\usepackage{todonotes}

\newtheorem{lemma}{Lemma}[section]

\newtheorem{corollary}[lemma]{Corollary}
\newtheorem{proposition}[lemma]{Proposition}
\newtheorem{claim}[lemma]{Claim}
\newtheorem{theorem}[lemma]{Theorem}
\theoremstyle{definition}
\newtheorem{definition}[lemma]{Definition}

\newtheorem{remark}[lemma]{Remark}

\DeclareMathOperator{\supp}{supp}
\DeclareMathOperator{\Alg}{Alg} 
\DeclareMathOperator{\Tr}{Tr} 
\DeclareMathOperator{\esssup}{ess\,sup} 
\DeclareMathOperator{\I}{{\mathbb{I}}}
\DeclareMathOperator{\mbbP}{{\mathbb{P}}}

\DeclareMathOperator{\U}{{\mathcal{U}}}
\DeclareMathOperator{\Gr}{{\mathcal{G}_*}}
\DeclareMathOperator{\Grr}{{\mathcal{G}_{**}}}
\DeclareMathOperator{\Hs}{{\mathcal{H}}}
\newcommand*{\field}[1]{\mathbf{#1}} 
\newcommand*\Z{\field{Z}} 
\newcommand*\N{\field{N}} 
\newcommand*\R{\field{R}} 
\newcommand*\C{\field{C}} 
\newcommand*\Un{\mathcal{U}} 

\newcommand*\eps{\varepsilon}

\makeatletter
\DeclareRobustCommand{\cev}[1]{%
  {\mathpalette\do@cev{#1}}%
}
\newcommand{\do@cev}[2]{%
  \vbox{\offinterlineskip
    \sbox\z@{$\m@th#1 x$}%
    \ialign{##\cr
      \hidewidth\reflectbox{$\m@th#1\vec{}\mkern4mu$}\hidewidth\cr
      \noalign{\kern-\ht\z@}
      $\m@th#1#2$\cr
    }%
  }%
}
\makeatother
\usepackage{hyperref}

\title{Localization of eigenfunctions in amenable unimodular random networks}

\author{Georgii Veprev}

\date{\today}

\begin{document}

\maketitle

\begin{abstract}
    For an amenable unimodular random rooted network, we show that the presence of a positive point mass in the expected spectral measure implies that, with positive probability, there exists an eigenfunction with finite support. 
\end{abstract}

\section{Introduction}

The Kaplansky’s zero divisor conjecture states that for a torsion-free group~$\Gamma$, there exist no zero divisors in the group algebra~$\C\Gamma$. That is, for any $\alpha \not = 0$ and~$\beta \not = 0$ in~$\C\Gamma$, the product $\alpha\beta$ is non-zero. It was shown by~Elek in~\cite{El} that when~$\Gamma$ is amenable, this conjecture is equivalent to its analytic version where~$\beta$ is taken in~$l^2(\Gamma)$. Namely, he showed that for any amenable group~$\Gamma$ and any non-zero $\alpha \in \C\Gamma$, if there exists non-zero~$\beta \in l^2(\Gamma)$ such that $\alpha \beta = 0$ then there exists such $\beta\in \C\Gamma$. One special case of particular interest is the case of an amenable group~$\Gamma$ generated by a finite symmetric set $S = \{s_1, s_2, \ldots, s_n\}$ and~$A = {s_1 + s_2 + \ldots + s_n} \in \C\Gamma$, the adjacency operator on the Cayley graph~$Cay(\Gamma,S)$. The result of Elek applied to $\alpha = A - \lambda$, $\lambda \in  \C$, then says that for amenable groups, the existence of an eigenfunction of the adjacency operator implies the existence of an eigenfunction with finite support. Moreover, it can be shown that for any $\lambda$ there exists a basis of the $\lambda$-eigenspace consisting of eigenfunctions with finite support. In other words, for amenable groups, all eigenfunctions are generated by finitely supported ones and, hence, the discrete part of the spectrum is described by the local geometry of the Cayley graph. Let us mention that in a group with non-trivial torsion, the adjacency operator can possess eigenfunctions and, moreover, the spectrum can be pure point as shown by Grigorchuk and \.Zuk in~\cite{GrZ}.

The goal of this paper is to prove a similar localization principle in the context of unimodular random rooted graphs and networks (that is, graphs with certain marks assigned to their edges). Such graphs appear naturally in statistical mechanics (e.g., as percolation clusters), geometric group theory (e.g., as Cayley graphs or Schreier graphs corresponding to an invariant random subgroup), and various other areas. For a survey on unimodular graphs and networks see~\cite{AL07}. 
We consider operators on rooted networks that depend only on a neighbourhood of a vertex of a fixed radius. These operators, which we call finite-range operators, are a natural counterpart of the group algebra in the graph theoretic setting. Examples of such operators include the adjacency and Laplace operators, as well as Schrödinger operators. We show in Theorem~\ref{theorem:main} (more precisely, in Corollary~\ref{theorem:amenable}) that in an amenable unimodular random rooted network, the existence of a point mass in the expected spectral measure of a finite-range operator implies the existence of a finitely supported eigenfunction with the same eigenvalue. Under a slightly stronger assumption that we call {F\o lner} balancedness, we show in Theorem~\ref{theorem:balanced} that the von Neumann dimension of an eigenspace can be approximated by the dimension of the space of eigenfunctions supported on a F\o lner set in a graph. This implies that every eigenspace has a basis consisting of finitely supported eigenfunctions (see Corollary~\ref{corollary:basis}) and, hence, the discrete part of the expected spectral measure is described by finitely supported eigenfunctions. We also show in Corollary~\ref{corollary:sofic1} that for a given $\lambda\in \C$, the von~Neumann dimension of a $\lambda$-eigenspace of a finite-range operator is a continuous function on the space of {F\o lner} balanced networks.

The localization principle was proved earlier by Lenz and Veseli{\'c}~\cite{LV}, and Grigorchuk and Pittet~\cite{GP} for Sch\"odinger operators on amenable quasi-homogeneous graphs.  An amenable quasi-homogeneous graph admits an action of an amenable group with finitely many orbits and, hence, gives rise to an amenable unimodular random graph supported on finitely many non-isomorphic rootings. Sch\"odinger operators, whether random or not, can be realised as range one operators on a network, where the potential is represented by the weights of the loop at each vertex. Hence, Theorem~\ref{theorem:main} generalises these results. Moreover, Theorem~\ref{theorem:balanced} gives a possible way to compute the exact values of the atoms in the density of states of the corresponding random graph or network. Note that in the case of quasi-homogeneous graphs, the density of states is a finite convex combination of the Kesten spectral measures corresponding to non-isomorphic rootings. Let us also note that our setup includes operators that are not necessarily self-adjoint or bounded. 

An interesting source of examples of unimodular random rooted graphs are random Schreier graphs arising from invariant random subgroups of finitely generated groups. Recall that an invariant random subgroup is a distribution on the space of subgroups that is invariant under the conjugation action of the ambient group.
For self-similar Schreier graphs of certain finitely generated groups acting on a rooted tree, it was shown earlier that the density of states is pure point, and there is a basis consisting of finitely supported eigenfunctions, see~\cite{BGJRT, GNP, Tep_serpinski}.  
In contrast to this, in Section~\ref{section:long-range}, we apply our Theorem~\ref{theorem:main} to a family of unimodular random rooted graphs that we call long-range graphs and show the continuity of the expected spectral measure of the adjacency operator (see Theorem~\ref{theorem:long-range}). Such graphs include so-called $\omega$-periodic graphs defined in~\cite{BH} and studied in~\cite{BCDN} as Schreier graphs of a certain automaton group.

The papers~\cite{El},~\cite{GP}, and~\cite{LV} use as the main tool the von Neumann algebras of the corresponding homogeneous spaces. We will use similar techniques with the help of the von Neumann trace on unimodular random rooted graphs and networks, see, e.g,~\cite{AL07}. There are several difficulties specific to the realm of unimodular random graphs that we had to overcome on the way. 

\subsection*{Acknowledgments} The author is sincerely grateful to his advisor Tatiana~Nagnibeda for her support and guidance during the work on this paper. 

\subsection{Unimodular random rooted graphs and networks}

In this paper, we work with rooted graphs and networks. To set up the necessary notation that we will use later on we follow a survey~\cite{AL07} by Aldous and Lyons. We consider locally finite connected graphs $G = (V,E)$ where~$V = V(G)$ is the set of vertices and~$E = (G)$ is the set of edges with possible multiple edges and loops. A root~$o \in G$ in a graph is a chosen vertex. Each edge~$e$ has two (possibly the same) associated vertices incident to it. The \emph{mark space}~$\Xi$ is an uncountable complete separable metric space, which one can assume to be~$\N^\N$. A \emph{rooted network} is a rooted graph~$(G,o)$ together with the marking (labeling) map~$c$ that assigns a label from $\Xi$ to every vertex, and a pair of labels for every edge corresponding to incident vertices. One can interpret these labels as orientation of edges as well as their colours or weights. Naturally we can consider rooted graphs as networks whose marks are all equal to a fixed special mark. Let~$\Gr$ be the space of all isomorphism classes of rooted networks. We consider the local topology on~$\Gr$ that can be defined by the following metric. For two rooted networks~$(G_1,o_1)$ and~$(G_2, o_2)$ let~$r$ be the maximal positive~$t$ such that there is a graph isomorphism~$f$ between $[t]$-neighbourhoods~$B_{[t]}(G_1,o_1)$ and~$B_{[t]}(G_2,o_2)$ that preserves the root and the corresponding marks are at distance at most~$\frac{1}{t}$ in~$\Xi$. The distance between~$(G_1,o_1)$ and~$(G_2, o_2)$ is the value~$\frac{1}{1 + r}$. This metric is separable and complete on~$\Gr$.
Note that a sequence $(G_n, o_n)$ of rooted networks with marks taken from a discrete subspace of~$\Xi$ converges if and only if for any radius~$R$ the balls $B_R(G_n,o_n)$ are graph-isomorphic for all large enough~$n$.  

One can choose in each element of~$\Gr$, that is, in an isomorphism class of rooted networks, its \emph{canonical representative}. The vertex set of this representative is fixed to be~$\N$ or, if the network is finite of size~$N$, an interval $\{0, \ldots, N -1\}$, and the root is set to be at~$0$. This choice can be made continuous on~$\Gr$. See~\cite{AL07} for details.

A \emph{random rooted network} is a Borel probability measure~$\rho$ on~$\Gr$. We consider weak-* topology on random rooted networks. A random rooted network~$\rho$ is called \emph{unimodular} if it satisfies the following condition called the mass-transport principle. Let~$\Grr$ be the space of isomorphism classes of networks with two distinguished vertices and~$f$ a Borel non-negative function on~$\Grr$. Then
\begin{equation}\label{eq_mass_transport}
    \int\limits_{\Gr} \sum_{x \in V(G)} f(G,o, x) =  \int\limits_{\Gr} \sum_{x \in V(G)} f(G, x, o).
\end{equation}
We denote by~$\Un$ the space of all unimodular random rooted networks. We will always assume that the random rooted network we work with is unimodular. Also, we will often omit "rooted" when referring to a measure in~$\Un$.

\subsection{Spectral measures and von Neumann trace}
Let~$\rho \in \Un$ be a unimodular random rooted network.
We consider the following Hilbert space 
$$
\mathcal{H} = \int\limits^{\oplus}_{\Gr} l^2 (V(G)) d \rho(G,o).
$$
The elements of~$\Hs$ are~$\rho$-measurable functions $f \colon \Gr \to l^2(\N)$ that assign to each isomorphism class of a rooted network a square summable function on the vertices of its canonical representative and such that 
$$
||f||^2_{\Hs} = \int\limits_{\Gr} ||f_{G,o}||^2 d \rho(G,o) < +\infty.
$$
We will write $f= \int^\oplus f_{G,o} d\rho$. Once a rooted network $(G,o)$ is understood, we will refer to~$f_{G,o}$ as a function in $l^2(V(G))$. 

Let $\{D_{G,o}\}$ be a measurable assignment of bounded linear operators where each $D_{G,o}$ acts on $l^2(V(G))$. This assignment gives rise to the following operator on $\mathcal{H}$
$$
D^\rho \colon \int\limits^{\oplus}_{\Gr} f_{G,o} d \rho(G,o) \mapsto \int\limits^{\oplus}_{\Gr} D_{G,o} f_{{G,o}} d \rho(G,o).
$$
The norm of $D^\rho$ is given by
$$
||D^\rho||_\mathcal{H} = \esssup ||D_{G,o}||_{l^2(V(G))},
$$
hence, $D^\rho$ is bounded if and only if $\rho$-almost all $||D_{G,o}||$ are uniformly bounded. 
Let $\Alg = \Alg(\rho)$ be the space of equivariant operators, that is, the operators invariant under isometries and changes of the root. This means that for any two networks~$G_1$ and $G_2$ and an (unrooted) isomorphism $\phi \colon G_1 \to G_2$
$$
\langle D_{G_1, o_1} \mathbb I _x , \mathbb I_y \rangle = 
\langle D_{G_2, o_2} \mathbb I_{\phi(x)} , \mathbb I_{\phi(y)} \rangle,
$$
where~$o_1, x, y \in V(G_1)$, $o_2 \in V(G_2)$, and~$\I_x$~is the indicator function of a vertex~$x$. In particular, the operator $D_{G,o}$ does not depend on the choice of the root~$o$ in~$G$, and we will write $D_G$ for simplicity in what follows. The space $\Alg$ forms a von Neumann algebra with a trace defined by the following functional 
$$
\Tr (D) = \int\limits_{\Gr} \langle D_{G} \mathbb I _o , \mathbb I_o \rangle d \rho(G,o).
$$
A closed subspace~$W \subset \mathcal{H}$ is called invariant if the corresponding projection~$P_W$ belongs to $\Alg$.
For an invariant subspace $W \subset \mathcal{H}$ its von Neumann dimension is the following value
$$
\dim_\rho W = \Tr(P_W).
$$
The following proposition appears to be folklore. Due to its importance for the present paper we include the proof here for completeness. 
\begin{proposition}
    The trace~$\Tr$ is faithful. That is, for any $A \in \Alg$ the equality $\Tr(A^*A) = 0$ implies~$A =0$.
\end{proposition}
\begin{proof}
    We have 
    $$
    0 = \Tr(A^*A) = \int \langle A_G^*A_G \mathbb I_o, \mathbb I_o\rangle d \rho(G,o) = \int \langle A_G \mathbb I_o, A_G \mathbb I_o\rangle d \rho(G,o).
    $$
    Therefore, $A_G \mathbb I_o = 0$ $\rho$-almost surely. Now let $f(G,o,x) = \langle A_G^*A_G \mathbb I _x, \mathbb I_x\rangle$. It is a Borel non-negative function which is invariant under rooted isomorphisms. The mass transport principle gives 
    $$
    \int \sum_{x \in G} \langle A_G^*A_G \mathbb I_x, \mathbb I_x\rangle d \rho(G,o)=
    \int \sum_{x \in G} \langle A_G^*A_G \mathbb I_o, \mathbb I_o\rangle d \rho(G,o) = 0.
    $$
    Then, for $\rho$-almost every~$(G,o)$ the and every~$x \in G$, $A_G\mathbb I_x =0$, hence, $A_G = 0$.
\end{proof}
Applying this to the projection operator of an invariant subspace~$W \in \mathcal{H}$, we obtain that $\dim_\rho W = 0$ if and only if~$W =0$.

\subsection{Operators of finite range}
Let~$D \in \Alg$ be a bounded operator on~$\mathcal{H}$ such that $\rho$-almost surely~$D_G$ is normal.  Let $E_{G}^D$ be the projection-valued spectral measure of the operator~$D_G$ and let $\mu_{G, o}^D = \langle E_{G}^D \I_o, \I_o \rangle$ be the scalar-valued spectral measure of~$D_G$ with respect to the indicator function of the root~$o \in G$. 
The \emph{expected spectral measure} (also known as the \emph{density of states}) of~$\rho$ for the operator~$D$ is 
$$
\mu_\rho^D = \int_{\mathcal{G_*}} \mu_{G,o}^D d \rho(G,o)
$$
Note that for any~$\lambda \in \C$ 
\begin{equation}\label{eq_dimatom}
    \mu_\rho^D(\lambda) = \dim_\rho E_\lambda^D, 
\end{equation}
where $E_\lambda^D \subset \mathcal{H}$ is the subspace of $\lambda$-eigenfunctions of~$D$. In what follows, we will often omit the indices~$D$ and~$\rho$ when referring to the expected spectral measure or the $\lambda$-eigenspace if the measure~$\rho$ and the operator~$D$ are understood. 

 Let us note the spectrum of equivariant normal finite-range operator does not depend on the choice of the root and is, in a proper sense, a measurable function on~$\Gr$. Hence if~$\rho \in \U$ is extremal, that is, can not be decomposed into a convex combination of distinct measures in~$\U$, then the spectrum of~$D$ is almost surely the same (see Theorem~4.7 in~\cite{AL07}). 

One can define the expected spectral measure for certain unbounded densely defined operators as long as the spectral theorem is applicable for $\rho$-almost every~$(G, o)$. A basic example of such an operator is the adjacency operator of a random graph that does not have a uniform bound on its degrees. The adjacency operator is always $\rho$-almost surely essentially self-adjoint if~$\rho$ is unimodular (see~\cite{bordenave2017spectrum}). In this paper, we consider a more general class of operators that depend only on some neighbourhood of a vertex. Let $a\colon \Grr \to \C$ be a continuous function such that~$a (G, x, y) = 0$ for any~$x$ and~$y$ in~$G$ at distance more than some~$r \in \N$. For a network~$G$, $x \in G$, and a function $\phi \in l^2(V(G))$ we define 
$$
D_G \phi (x) = \sum_{y \in V(G)} a(G,x,y) \phi(y).
$$
We will moreover assume that the value $a(G,o,x)$ depends only on the~$r$-neighbourhood~$B_r(G,o)$ and not on the rest of the network~$G$. We will call such operators \emph{operators of finite range~$r$} (also known as \emph{local operators}\footnote{This latter assumption of finite dependences of the coefficients can often be omitted in our arguments. The main use of this condition is Corollary~\ref{corollary:approximation}, where we consider the same eigenfunctions on different random networks.}, see~\cite{bordenave2025}). If 
$$
\esssup_\rho \sum_{x \in V(G)}|a(G,o,x)|
$$
is finite then we can view~$D = \{D_G\}$ as an element of~$\Alg(\rho)$. If, moreover,~$a$ is symmetric, meaning that $a(G,o,x) = \overline{a(G,x,o)}$ for every $o,x \in G$, then the operator~$D$ is self-adjoint. If~$D_G$ is normal for~$\rho$ almost every~$(G,o)$ then~$D$ is a normal operator on $\mathcal{H}$. 

In the case when $\sum_{x \in V(G)}|a(G,o,x)|$ is not uniformly bounded the operator~$D_G$ is still defined on any function $f \in l^2(V(G))$ with finite support. If~$a$~is symmetric and real, then~$D$ is almost surely essentially self-adjoint~\cite{bordenave2017spectrum}, and one can, hence, define its density of states. 
Let us note that in our setting the operator~$D_G$  is defined for all networks~$(G,o) \in \Gr$.

\subsubsection{Eigenspaces of operators of finite range}

As mentioned above, in order to define the spectral measure of an operator~$D$ on~$\mathcal{H}$, one has to assume that the spectral theorem can be applied to this operator. Although this can be done for a large class of operators, the analysis of eigenfunctions that we do in this paper is possible in a greater generality when the spectral theorem is not applicable. In fact, as it comes to the operator, only the assumption of the finite range will play a role. 
Let~$D$ be an operator of finite range and~$\lambda \in \C$. The $\lambda$-eigenspace of~$D$ is the subspace of~$\mathcal{H}$ defined as
$$
E_\lambda^D = \int\limits^\oplus_{\Gr} E_\lambda^D(G)d \rho,
$$
where~$E^D_\lambda(G)\subset l^2(G)$ is the $\lambda$-eigenspace of~$D_G$ on~$G$. Since the operator~$D$ is of finite range, the subspace~$E_\lambda^D$ is closed and invariant. 
Clearly~$\lambda$ is an eigenvalue for~$D$ on~$\mathcal{H}$ if and only if $E_\lambda^D$ is non-empty, that is, $\dim_\rho(E_\lambda^D) > 0$. In view of formula~\eqref{eq_dimatom}, the von Neumann dimension $\dim_\rho(E_\lambda^D)$ will be our main value of interest. 

\subsection{Amenability}

\subsubsection{Amenable random networks}\label{subsection_amen}

There exist several natural notions of amenability for unimodular random rooted graphs and networks (see, e.g.,~\cite{AL07, Ka1997amenability}). Most of these conditions are equivalent under the assumption of finite expected degree. The one that is the most convenient for us in the context of this paper is the one proposed in~\cite{AL07} as it deals with a kind of F\o lner condition. 

For a subset $A\subset G $, we denote by $\partial A$ its inner vertex boundary, that is, the set of points $x \in A$ that are adjacent to a vertex of~$G$ outside~$A$. Similarly, for a non-negative integer number~$r$ we will denote by $\partial_r A$ the set of points $x \in A$ that are at (edge) distance at most~$r$ from~$G \setminus A$. In particular, $\partial_0 A = \emptyset$ and $\partial_1 A = \partial A$. We will also consider the \emph{edge-boundary $\partial_E A$}, that is, the set of edges which have one end in~$A$ and the other outside~$A$. 
A \emph{percolation} on~$\rho \in \U$ is a unimodular random rooted network~$\rho^\prime \in \U(\Xi \times\Xi)$ together with a subset $\Xi_0 \subset\Xi$ such that the projection~$\pi \colon \U(\Xi \times\Xi) \to \U(\Xi)$ that forgets the second component of the marks maps~$\rho^\prime$ to~$\rho$. One can, of course, identify~$\Xi \times\Xi$ with a subset of~$\Xi$ and consider~$\rho^\prime$ as an element of~$\U(\Xi)$.  The~$\Xi_0$-component of a vertex~$x$ is the set of vertices reachable from~$x$ by edges whose second marks belong to~$\Xi_0$. We will denote by~$K(\Xi_0)$ the~$\Xi_0$-component of the root. Let~$FC(\rho)$ be the set of percolations on~$\rho$ that have only finite components almost surely. The \emph{isoperimetric constant} of~$\rho$ is the following value 
\begin{equation}\label{eq_amenability}
\imath_E(\rho) = \inf \int\frac{|\partial_E K(\Xi_0)|}{|K(\Xi_0)|} d \rho^\prime(G,o), 
\end{equation}

where $(\rho^\prime, \Xi_0)$ ranges over all percolations in $FC(\rho)$. Unimodular random rooted network~$\rho$ is called \emph{amenable} if~$\imath = 0$. In other words, $\rho$~is amenable if there exists a sequence of percolations whose clusters of the origin yield a F\o lner sequence.
One can also define the isoperimetric constant~$\imath_r$ for a positive integer~$r$ by replacing the edge-boundary~$\partial_E$ in formula~\eqref{eq_amenability} with the $r$-boundary~$\partial_r$. For a unimodular~$\rho$ with finite expected degree, the amenability condition $\imath_E = 0$ is equivalent to $\imath_r = 0$ for any positive integer~$r$ (see the proof of Proposition~\ref{proposition:rFolnler} in the following section).

\subsubsection{Balanced F\o lner sequences}

The definition of an amenable unimodular random network above has a disadvantage that the F\o lner sets that appear there are, in a sense, external to the realisations of the random graphs and come from different percolation processes on them. In this paper, we will use a different approach focusing on the F\o lner sequences inside the rooted networks and their properties.
Suppose that $\rho \in \mathcal{U}$ is such that for every positive integer~$n$ there exists a measurable assignment~$\{F_n(G,o)\}$ of finite subsets of~$V(G)$ such that 
\begin{equation}\label{eq:Folner}
    \lim_{n \to \infty} \int \frac{|\partial F_n (G,o)|}{|F_n(G,o)|} d\rho= 0.
\end{equation}
Naturally, we call such a sequence a \emph{F\o lner sequence} for~$\rho$. Note that the existence of a F\o lner sequence implies leaf-wise ($\rho$-almost surely) amenability of the graph. 
Different choice of the type of the boundary leads to several other versions of the F\o lner condition. We call a sequence~$\{F_n\}$ \emph{$r$-F\o lner} if the convergence~\eqref{eq:Folner} holds when~$\partial$ is replaced by~$\partial_r$ and \emph{edge-F\o lner} if~$\partial$ is replaced by~$\partial_E$. 
For graphs of uniformly bounded degree, this clearly does not give anything new. However, this does make a difference for locally finite graphs without uniform bounds, which we would like to keep under consideration.

For a measurable assignment of a subset $A(G,o) \subset G$, the function~$g_A$ is defined as
\begin{equation}\label{definition:g_function}
    g_A(G,o) = \sum\limits_{o \in A(G, x)} \frac{1}{|A(G,x)|}.
\end{equation}
The use of this function is as follows. Suppose~$\phi$ is a non-negative measurable function on $\mathcal{G_*}$. Then the average value of~$\phi$ over~$A$ can be rewritten via the mass transport principle 
\begin{multline}\label{eq:avfunction}
    \int \frac{1}{|A(G,o)|} \sum_{x \in A(G,o)} \phi(G,x) d\rho(G,o) = \\ \int \phi(G,o) \sum_{o \in A(G,x)} \frac{1}{|A(G,x)|} d\rho(G,o) = \int \phi g_A d\rho.
\end{multline}
In particular, we have $\int g_A d \rho= 1$. Let us also note that $\esssup g_A$ is the norm of the operator acting on $L^1(\Gr, \rho)$ by taking the average value over~$A$. We would like to have certain control over averages over the F\o lner sets in our graphs. As formula~\eqref{eq:avfunction} shows,  in order to do that one should impose regularity conditions on the function~$g$. This motivates the following definition.
\begin{definition}
    We call a unimodular random network~$\rho$ ($r$-)\emph{{F\o lner} balanced} if there exists a measurable ($r$-)F\o lner sequence $F_n(G,o)$ such that for $\rho$-almost every~$(G,o)$
    $$
    \lim_{n \to \infty} g_{F_n}(G,o)  = 1. 
    $$
    We call $\rho$ \emph{weakly} ($r$-)\emph{{F\o lner} balanced} if there exists a measurable ($r$-)F\o lner sequence $F_n(G,o)$ and constant $c > 0$ such that for $\rho$-almost every $(G,o)$ 
    $$
     \liminf_{n\to \infty}g_{F_n}(G,o) > c. 
    $$  
     Similarly, we define~\emph{edge-{F\o lner} balanced} and \emph{weakly edge-{F\o lner} balanced} unimodular random networks.
\end{definition}

The following simple proposition allows one to reduce the $r$-boundary to the standard boundary of a F\o lner set under certain regularity conditions. 
\begin{proposition}\label{proposition:rFolnler}
    Suppose that $\rho \in \mathcal{U}$ has a finite expected degree of the root. Let $\{F_n\}$ be a F\o lner sequence such that $g_{F_n}< C$ almost surely for all~$n$. Then for any positive integer~$r$
    $$
    \lim_{n \to \infty} \int \frac{|\partial_r F_n (G,o)|}{|F_n(G,o)|} d\rho= 0.
    $$
\end{proposition}
\begin{proof}
    Let us first prove the claim for~$r = 2$. Clearly $|\partial_2 A| \le \sum_{x \in \partial A} \deg x$. We have 
    \begin{multline}\label{eq:50506}
         \int \frac{|\partial_2 F_n (G,o)|}{|F_n(G,o)|} d\rho \le 
         \int \frac{\sum_{x \in \partial F_n(G,o)} \deg x}{|F_n(G,o)|} d\rho = \\
         \int \deg o\sum_{o \in \partial F_n(G,x)} \frac{1}{|F_n(G,x)|} d\rho = 
         \int \deg (o) h_n(G,o) d\rho,
    \end{multline}
    where $0\le h_n(G,o) \le g_{F_n}(G,o) < C$ and $\int h_n$ tends to zero in~$n$. This implies that the last integral in formula~\eqref{eq:50506} tends to zero as well. 
One can now proceed by induction in~$r$.
\end{proof}
Note that the conclusion of Proposition~\ref{proposition:rFolnler} holds if~$||g_{F_n}||_\infty$ is uniformly bounded and $||\deg||_1 < \infty$ or if $||\deg ||_\infty < \infty$ and $||g_{F_n}||_1$ is uniformly bounded (which is automatic as $||g_{A}||_1 = 1$ for any~$A$). A slightly more general sufficient condition can be deduced from the proof of Proposition~\ref{proposition:rFolnler} and the Holder inequality: it is enough to require that $||\deg||_p$ is finite and $||g_{F_n}||_q$ is uniformly bounded for some~$p, q \in [1, +\infty]$ such that $\frac{1}{p} + \frac{1}{q} = 1$. Let us also note that the identical  proof applies to the {edge-boundary}~$\partial_E F_n(G,o)$.
Similar computation to inequality~\eqref{eq:50506} shows that the condition~$\imath_E= 0$ that defines amenability (see Section~\ref{subsection_amen}), and the condition~$\imath_r =0$ are equivalent for unimodular random networks with finite expected degree and any positive integer~$r$.

Another fact that will be of use later on is that we can choose our ($r$-)F\o lner sequences to be continuous functions for every~$n$. To be precise, we prove the following proposition. 
\begin{proposition}\label{proposition:Folner_continuous}
Let~$\{\tilde F_n\}$ is a measurable sequence of subsets satisfying some of the above properties (($r$-)F\o lner, (weakly) balanced). Then there exists a measurable sequence~$\{F_n\}$ and a sequence of positive integers~$d_n$ such that for every~$(G,o) \in \mathcal{G}_*$ we have $F_n(G,o) \subset B_{d_n}(G,o)$, the subset $F_n(G,o)$ depends only on~$B_{d_n}(G,o)$, and $\{F_n\}$ satisfies the same properties as $\tilde F_n$. 
\end{proposition}

\begin{proof}
    First, we replace $\tilde F_n$ by~$F_n$ defined as 
    $$ F(G,o)= \tilde F_n(G,o) \text{ if } \tilde F_n(G,o) \subset B_{d_n}(G,o); \quad 
    F_n(G,o) = \{o\} \text{ otherwise},$$ 
    where~$d_n$ is a sufficiently large constant to be defined later. Then for any $\eps_n > 0$ and any sufficiently large~$d_n$
    \begin{equation}\label{eq:approxfolner}
        \mbbP \left( F_n(G,o) \not = \tilde F_n(G,o) \right) < \eps_n.
    \end{equation}

    We can also assume, preserving inequality~\eqref{eq:approxfolner}, that either the degrees of all vertices in $B_{d_n}$ are no more than~$d^\prime_n$, which depends on~$d_n$, or $F_n(G,o) = \{o\}$. Now, we can partition~$\mathcal{G_*}$ by the values of~$F_n(G,o)$ in~$B_{d_n}$ viewed as subsets of~$\N$ in the canonical representative of~$(G,o)$. This is a finite partition into Borel subsets. This partition can be approximated by a finite partition into the cylinder sets. This defines a function satisfying~\eqref{eq:approxfolner} and which depends only on the neighbourhood of the root of size~$R$ for some finite~$R$. We can then redefine~$d_n$ to be the maximum of~$d_n$ and the number~$R$. 

    Given relation~\eqref{eq:approxfolner} we have for any positive~$r$
    $$
     \int \frac{|\partial_r F_n (G,o)|}{|F_n(G,o)|} d\rho \le \int \frac{|\partial_r\tilde F_n (G,o)|}{|\tilde F_n(G,o)|} d\rho + \eps_n.
    $$
    Moreover~for any positive $\phi \in L^\infty(\mathcal{G}_*, \rho)$
    $$
    \int g_{F_n}(G,o)\phi(G,o) d\rho  = \int \frac{1}{|F_n(G,o)|} \sum_{x \in F_n(G,o)} \phi(G,x) d\rho.
    $$
    And, hence, 
    \begin{equation}\label{eq_contfolner}
        \left| \int g_{F_n}(G,o)\phi(G,o) d\rho  - \int g_{\tilde F_n}(G,o)\phi(G,o) d\rho \right| \le 2\eps_n ||\phi||_\infty,
    \end{equation}
    as any average of~$\phi$ does not exceed $||\phi||_\infty$. It is, therefore, enough to require~$\eps_n$ to decay to zero for preserving the property of being~$r$-F\o lner. If this decay is sufficiently fast, then the sequence~$F_n$ inherits the property of being balanced or weakly balanced. Indeed, assume~$\{\tilde F_n\}$ is balanced. One can choose $\phi (G,o)$ to be the sign of the difference $g_{F_n}(G,o) - g_{\tilde F_n}(G,o)$. Inequality~\ref{eq_contfolner} then implies convergence in $L^1(\Gr, \rho)$ of $g_{\tilde F_n}$ to~$1$. Hence, if the sequence~$\eps_n$ is summable, then we have convergence almost surely, that is, balancedness of~$\{F_n\}$. Similarly, one can treat the case of weakly balanced~$\{\tilde F_n\}$.
\end{proof}

Note that Proposition~\ref{proposition:Folner_continuous} implies, in particular, that the function $g_{F_n}$ is continuous on~$\Gr$ and depends only on $2d_n$-neighbourhood of the root.

Finally, we show that the {F\o lner} balancedness (more precisely, edge-{F\o lner} balancedness) implies the amenability of the random network in the sense of Section~\ref{subsection_amen}.

\begin{proposition}\label{proposition:balance_implies_amen}
    Suppose that~$\rho$ is an edge-{F\o lner} balanced measure on $\mathcal{G}_*$ of finite expected degree. Then~$\rho$ is amenable. 
\end{proposition}

\begin{proof}
    Due to Theorem~8.5 in~\cite{AL07} it suffices to find a sequence of measurable functions $\lambda_n \colon \Grr \to [0,1]$ such that for $\rho$~almost any~$(G,o)$
    \begin{equation}\label{eq_amen1}
        \sum_{x \in V(G)} \lambda_n(G, o, x) = 1,
    \end{equation}
    and
    \begin{equation}\label{eq_amen2}
        \lim_{n\to \infty} \int \sum_{y \sim o} \sum_{x \in V(G)} |\lambda_n(G,o,x) - \lambda_n(G,y,x)| d \rho(G,o) = 0.
    \end{equation}
    Let $\{F_n\}$ be a balanced edge-F\o lner sequence for~$\rho$. We put
    $$
    \lambda_n(G,o, x) = \frac{1}{g_{F_n}(G,o)} \cdot\frac{\I(o \in F_n(G,x))}{ |F_n(G,x)|}.
    $$
    Condition~\eqref{eq_amen1} is automatically satisfied due to the definition of $g_n$. Let us show that the second condition~\eqref{eq_amen2} holds for our choice of $\lambda_n$-s. To be precise, we will show that it is satisfied for a subsequence. First, we show that
    \begin{equation}\label{eq_amen3}
        \lim_{n\to \infty} \int \sum_{y \sim o} \sum_{x \in V(G)} |g_{F_n}(G,o)\lambda_n(G,o,x) - g_{F_n} (G,y)\lambda_n(G,y,x)| d \rho(G,o) = 0.
    \end{equation}
    Indeed, 
    \begin{multline}\label{eq_am4567567}
        \int \sum_{y\sim o}\sum_{x \in V(G)} \left| \frac{\I(o \in F_n(G,x))}{ |F_n(G,x)|} - \frac{\I(y \in F_n(G,x))}{ |F_n(G,x)|} \right| d \rho(G,o) = \\
        \int \sum_{y\sim o} \sum_{x\in V(G)} \frac{\I((o,y) \in \partial_E F_n(G,x)) + \I((y,o) \in \partial_E F_n(G,x))}{ |F_n(G,x)|} d \rho(G,o) = \\
        \int \sum_{x \in V(G)} \frac{2|\partial_E F_n(G,o)|}{|F_n(G,o)|} d \rho(G,o).
    \end{multline}
    The last equality holds due to the mass transport principle. As the sequence~$\{F_n\}$ is edge-F\o lner the result of computation~\eqref{eq_am4567567} tends to zero in~$n$. 
    Then we note that the sum 
    $$
    \sum_{y \sim o} \sum_{x \in V(G)} |\lambda_n(G,o,x) - \lambda_n(G,y,x)|
    $$
    is bounded almost surely by~$2\deg o$ as $\lambda_n(G,o, \cdot)$ is a probability measure on~$V(G)$ for every~$o \in V(G)$. Hence, it is enough to show that this sum almost surely converges to~$0$ as the degree function is integrable. We have 
    \begin{multline}
        \sum_{y \sim o} \sum_{x \in V(G)} |\lambda_n(G,o,x) - \lambda_n(G,y,x)| \le \\
        \frac{1}{g_{F_n(G,o)}}\sum_{y \sim o} \sum_{x \in V(G)} |g_{F_n}(G,o)\lambda_n(G,o,x) - g_{F_n} (G,y)\lambda_n(G,y,x)| + \\
        \frac{1}{g_{F_n(G,o)}} \sum_{y \sim o} \left|{g_{F_n}(G,y)}- {g_{F_n}(G,o)}\right| \sum_{x \in V(G)} \lambda_n(G,y,x).
    \end{multline}
    As $\{F_n\}$ is balanced, $g_n(G,o)$ tends to~$1$ almost surely. We showed that the expectation of the first sum converges to~$0$, hence, we can choose a subsequence such that is it converges to~$0$ almost surely. The sum $\sum_{x  \in V(G)} \lambda_n(G, y, x) = 1$, thus, we are left to estimate 
    \begin{multline}
        \sum_{y \sim o} \left|{g_{F_n}(G,y)}- {g_{F_n}(G,o)}\right| \le \\
        \sum_{y\sim o}\sum_{x \in V(G)} \left| \frac{\I(o \in F_n(G,x))}{ |F_n(G,x)|} - \frac{\I(y \in F_n(G,x))}{ |F_n(G,x)|} \right|.
    \end{multline}
    The expected value of this sum tends to~$0$ due to relation~\eqref{eq_amen3} and, therefore, we can again choose a subsequence to obtain a pointwise convergence to~$0$ almost surely.
\end{proof}

\section{Strong localization of eigenfunctions}

\subsection{The main results}

In this section we state and prove our main results on localization of eigenfunctions of finite-range operators as well as several corollaries of interest. 

\begin{theorem}\label{theorem:main}
    Let~$\rho \in \mathcal{U}$ be a weakly $r$-{F\o lner} balanced unimodular random network for some positive integer~$r$ and~$D$ be an operator of finite range~$r$. Suppose that there exists $\lambda\in \C$ such that~$\dim_\rho(E_\lambda^D) > 0$. Then with positive probability the network has a $\lambda$-eigenfunction with finite support.
\end{theorem}

\begin{corollary}
    If in Theorem~\ref{theorem:main} $\rho$ is extremal then the network has a finitely supported eigenfunction $\rho$-almost surely.
\end{corollary}

\begin{remark}
    One can easily show that the reverse implication of Theorem~\ref{theorem:main} always holds if the marks of the random network are taken from a discrete subspace. Suppose that for some $(G,o) \in \supp \rho \subset \mathcal{G}_*$  there exists a finitely supported $\lambda$-eigenfunction~$f$ for the operator~$D$. Then $\dim_\rho^D(E_\lambda) > 0$. Indeed, there is a positive integer~$N$ such that the support of~$f$ is contained in $B_N(G,o)$. The measure of all rooted networks that have $(N+r)$-neighbourhood isomorphic to $B_{N+r}(G,o)$, where~$r$ is the range of~$D$, is positive. Hence, the eigenspace $E_\lambda$ is non-empty with positive probability and, therefore, $\dim_\rho(E_\lambda) > 0$.   
\end{remark}

\begin{proof}[Proof of Theorem~\ref{theorem:main}]
    Let~$A(G,o) \subset V(G)$ be a measurable assignment of a finite subset in~$(G,o)$ and 
    let~$W$ be an { invariant} subspace of~$\mathcal{H}$ of the form 
    $$
    W = \int\limits^{\oplus}_{\Gr} W{(G,o)} d \rho(G,o)  = \int\limits^{\oplus}_{\Gr} W{(G)} d \rho(G,o).
    $$
    Let $P_W = \{P_{W(G)}\} \in \Alg$ be the orthogonal projection onto~$W$. Let us define the following quantity 
    $$
    \dim_A (W) = \int\limits_{\Gr} \frac{\sum_{x \in A(G,o)} \langle P_{W(G)} \mathbb I _x , \mathbb I _x\rangle}{|A(G,o)|} d \rho(G,o).
    $$
    The mass transport principle, more precisely, formula~\eqref{eq:avfunction} applied to the function~$\langle P_{W(G)} \mathbb I _x , \mathbb I _x\rangle$, gives
    $$
    \dim_A(W) = \int g_A(G,o) \langle P_W(G)\mathbb I_o, \mathbb I_o\rangle d \rho(G,o).
    $$
    Note that here we use the fact that~$P_W(G)$ does not depend on the choice of the root in~$G$, that is, the invariance of~$W$.
    
    Let $W_A(G,o)$ be the projection of $W(G)$ to the span $\langle \mathbb I_x \colon x\in A(G,o)\rangle$, that is, the subspace of functions supported on~$A(G,o)$. One has the following inclusion  
    $$
    W{(G)} \subset  W_A(G,o) \oplus \langle \mathbb I_x \colon x\not \in A(G,o)\rangle = Y(G,o).
    $$
    Therefore, 
    \begin{equation}\label{eq_proofthm1_1}
        \frac{\sum_{x \in A(G,o)} \langle P_{W(G)} \mathbb I _x , \mathbb I _x\rangle}{|A(G,o)|} \le 
    \frac{\sum_{x \in A(G,o)} \langle P_{Y(G,o)} \mathbb I _x , \mathbb I _x\rangle}{|A(G,o)|} =
    \frac{\dim W_A (G,o)}{|A(G,o)|}.
    \end{equation}

    Now let~$W$ be the space $E_\lambda$ of $\lambda$-eigenfunctions of~$D$ and~$A$ be an element of an $r$-F\o lner sequence $\{F_n\}_n$.
    Since~$\rho$ is weakly {F\o lner} balanced by assumption, we can choose this sequence such that there exists a constant~$c > 0$ satisfying the following bound almost surely for sufficiently large~$n$
    \begin{equation}\label{eq_proof1}
        g_{F_n}(G,o) > c.
    \end{equation}
    Therefore,
    \begin{multline}\label{eq_proof2}
        \liminf_{n \to \infty} \dim_{F_n} E_\lambda = 
        \liminf_{n \to \infty} \int g_{F_n}(G,o) \langle P_{E_\lambda(G)}\mathbb I_o, \mathbb I_o\rangle d \rho(G,o) \ge \\
        c \cdot \int \langle P_{E_\lambda(G)}\mathbb I_o, \mathbb I_o\rangle d \rho(G,o) = c \Tr(P_{E_\lambda}) = c \cdot \dim_\rho (E_\lambda) > 0.
    \end{multline}
    Here, we also use that~$g_{F_n}$ is non-negative almost everywhere. On the other hand, due to inequality~\eqref{eq_proofthm1_1} we have 
    \begin{equation}\label{eq_finproj}
        \dim_{F_n} E_\lambda \le 
        \int \frac{\dim {E_{\lambda,F_n}}(G,o)}{|F_n(G,o)|} d \rho (G,o),
    \end{equation}
    where $E_{\lambda, F_n}(G,o)$ is the projection of $E_\lambda(G)$ to $\langle \mathbb I_x \colon x\in F_n(G,o)\rangle$. 

    Since~$D$ has range~$r$, the space $E_{\lambda, F_n}(G,o)$ consists of functions~$f$ supported on~$F_n(G,o)$ which satisfy $Df(x) = \lambda f(x)$ for each~$x$ at distance at least~$r$ from the boundary of~$F_n(G,o)$. Suppose that two elements~$f_1$ and~$f_2$ of~$E_{\lambda,F_n}(G,o)$ coincide on $\partial_r F_n(G,o)$. Then their difference $f_1 - f_2$ extended by~$0$ outside~$F_n(G,o)$ is a $\lambda$-eigenfunction of~$D$ with finite support in~$F_n(G,o)$. Hence, either $(G,o)$ has a $\lambda$-eigenfunction with finite support in~$F_n(G,o)$ or 
    $$
    \dim {E_{\lambda,F_n}}(G,o) \le |\partial_r F_n(G,o)|.
    $$
    Let $p_\lambda$ be the probability of the event~$B_\lambda$ that consists of $(G,o)$ that have a $\lambda$-eigenfunction with finite support. Then
    \begin{multline*}
      \int\limits_{\Gr} \frac{\dim {E_{\lambda,F_n}}(G,o)}{|F_n(G,o)|} d \rho (G,o) \le 
    \int\limits_{B_\lambda} 1 d \rho (G,o) + 
    \int\limits_{\Gr \setminus B_\lambda} \frac{|\partial_r F_n(G,o)|}{|F_n(G,o)|} d \rho (G,o) \\ \le
    p_\lambda + \int\limits_{\Gr}\frac{|\partial_r F_n(G,o)|}{|F_n(G,o)|} d \rho(G,o).  
    \end{multline*}
    Together with equations~\eqref{eq_proof2}, \eqref{eq_finproj} this gives
    \begin{equation}\label{eq_proof3}
        c \cdot \dim_\rho (E_\lambda) \le 
        \liminf_{n \to \infty}\dim_{F_n} E_\lambda \le p_\lambda + \liminf_{n \to \infty} \int\frac{|\partial_r F_n(G,o)|}{|F_n(G,o)|} d \rho(G,o).
    \end{equation}
    As the sequence~$\{F_n\}$ is $r$-F\o lner, the last summand in formula~\eqref{eq_proof3} is zero and we conclude that $p_\lambda > 0$. 
    
    In the case of extremal~$\rho$, the conclusion follows as the property of having a $\lambda$-eigenfunction is invariant under unrooted isomorphisms of the network. Hence, the probability of this event is either~$0$ or~$1$ (see~\cite{AL07}) and, as it is positive, it has to be~$1$. 
\end{proof}

    The following are the immediate corollaries of Theorem~\ref{theorem:main}. 
    
\begin{corollary}\label{corollary:algebraic}
    Let~$\rho \in \mathcal{U}$ be weakly $r$-{F\o lner} balanced and~$D$ be an  operator with finite range~$r$ and rational coefficients. Then any~$\lambda$ such that $\dim_\rho(E^D_\lambda) > 0$ is an algebraic number. 
\end{corollary}

\begin{corollary}\label{corollary:sofic1}
    Suppose that the mark space~$\Xi$ is discrete and let a sequence of networks~$\{\rho_n\}$  converge in~$\U$ to a weakly $r$-{F\o lner} balanced $\rho \in \mathcal{U}$. And let~$D$ be an operator with finite range~$r$ such that $\dim_\rho (E^D_\lambda) > 0$ for some $\lambda \in \mathbb \C$. Then for all large enough~$n$ the network $\rho_n$ has a finitely supported $\lambda$-eigenfunction for~$D$ with positive probability and, in particular, $\dim_{\rho_n}(E^D_\lambda)> 0$. 
\end{corollary}

For (strongly) {F\o lner} balanced unimodular random networks, we can prove a refined version of Theorem~\ref{theorem:main}. We will show that not only finitely supported eigenfunctions exist but that the average dimension of the subspace of $\lambda$-eigenfunctions with finite support inside a F\o lner set converges to the value of~$\dim_\rho E_\lambda$. 
\begin{theorem}\label{theorem:balanced}
    Suppose that~$\rho$ is $r$-{F\o lner} balanced and $F_n(G,o)$ is a balanced $r$-F\o lner sequence. Let~$D$ be an operator of finite range~$r$. Then 
    $$
    \dim_\rho(E_\lambda) = \lim_{n \to \infty}\int \frac{ \dim \langle f \in E_\lambda(G) \mid \supp f \subset F_n(G,o)\rangle}{|F_n(G,o)|} d\rho(G,o).
    $$
\end{theorem}
\begin{proof}
    We will follow the proof of Theorem~\ref{theorem:main} with certain changes. Instead of inequality~\eqref{eq_proof1}, we now have for $\rho$-almost every rooted network~$(G,o)$
    $$
    \lim_{n \to \infty} g_{F_n}(G,o) = 1.
    $$
    Then the inequality~\eqref{eq_proof2} becomes
    \begin{multline}\label{eq45986009457}
        \lim_{n\to \infty} \dim_{F_n} E_\lambda = 
         \lim_{n\to \infty} \int g_{F_n}(G,o)\langle P_{E_\lambda(G)}\mathbb I_o, \mathbb I_o\rangle d \rho(G,o) = \\
        \int \langle P_{E_\lambda(G)}\mathbb I_o, \mathbb I_o\rangle d \rho(G,o) = 
    \dim_\rho (E_\lambda).
    \end{multline}
    Here we use that $g_{F_n}$ is non negative for every~$n$ with~$||g_{F_n}||_{1} =1$, and the function $\langle P_{E_\lambda(G)}\mathbb I_o, \mathbb I_o\rangle$ is non-negative as well and bounded from above by~$1$. 
    Let us then estimate the dimension of $E_{\lambda,F_n}(G,o)$. Let $E_{\lambda, F_n}^*(G,o) \le E_{\lambda,F_n}(G,o)$ be the subspace of $\lambda$-eigenfunctions with finite support inside~$F_n(G,o)$. The difference of any two functions in~$E_{\lambda,F_n}(G,o)$ that coincide on $\partial_r F_n$ belongs to $E_{\lambda, F_n}^*(G,o)$. Again, we extend this difference by zero outside~$F_n(G,o)$. Therefore, 
    \begin{equation}\label{eq4589456907}
        \dim E_{\lambda,F_n}(G,o) \le \dim
     E_{\lambda, F_n}^*(G,o) + |\partial_r F_n|.
    \end{equation}
    Combining relations~\eqref{eq_finproj}, \eqref{eq45986009457} and~\eqref{eq4589456907} we obtain 
    \begin{multline}
        \dim_\rho(E_\lambda) = 
        \lim_{n\to \infty} \dim_{F_n} E_\lambda \le 
        \lim_{n\to \infty} \int \frac{\dim {E_{\lambda,F_n}}(G,o)}{|F_n(G,o)|} d \rho (G,o) \le \\
        \lim_{n\to \infty} \int \frac{ \dim  E_{\lambda, F_n}^*(G,o)}{|F_n(G,o)|} d\rho(G,o), 
    \end{multline}
    where we use the $r$-F\o lner condition 
    $$
    \lim_{n\to\infty} \int \frac{|\partial_r F_n(G,o)|}{|F_n(G,o)|} d\rho(G,o) = 0
    $$
    to obtain the last inequality.
    
    Now we show the inequality in the other direction. By definition, $  E_\lambda(G) \ge E_{\lambda, F_n}^*(G,o)$ and, hence,
    \begin{multline}
        \dim_{F_n} E_\lambda = 
        \int \frac{\sum_{x \in F_n(G,o)}\langle P_{E_\lambda(G)} \mathbb{I}_x,\mathbb{I}_x \rangle}{|F_n(G,o)|}d\rho(G,o) \ge \\
        \int \frac{\sum_{x \in F_n(G,o)}\langle P_{E^*_{\lambda, F_n}(G,o)} \mathbb{I}_x,\mathbb{I}_x \rangle}{|F_n(G,o)|}d\rho(G,o) =
        \int \frac{\dim E^*_{\lambda, F_n}(G,o) }{|F_n(G,o)|}d\rho(G,o)
    \end{multline}
    Using relation~\eqref{eq45986009457} and the fact that 
    $$
    \dim E_{\lambda, F_n}^*(G,o) = \dim
    \langle f \in E_\lambda(G) \mid \supp f \subset F_n(G,o)\rangle, 
    $$
    we obtain the reverse inequality. 
\end{proof}

The following result was obtained in~\cite{ATV} for adjacency operators on graphs of bounded degree. We give an alternative poof of this result under the condition that the limiting measure is {F\o lner} balanced. Although our proof is not applicable in the non-amenable case, for {F\o lner} balanced graphs, it is more general, as we do not need the degree of the graphs and the coefficients of the operator to be bounded. In fact, we do not require any conditions on the coefficients of the operator~$D$ except the finite range. 
\begin{corollary}\label{corollary:approximation}
    Let~$D$ be an operator of finite range~$r$. If a sequence of networks~$\rho_n$ with discrete marks converges in~$\mathcal{U}$ to an $r$-{F\o lner} balanced~$\rho$, then 
    $$
    \lim_{n \to \infty}\dim_{\rho_n}(E_\lambda^D) = \dim_\rho (E_\lambda^D).
    $$
\end{corollary}

\begin{proof}
    As we will deal with different unimodular measures $\rho \in \U$, we will specify in our arguments the notation $\Tr_\rho$, $\dim_\rho$, and $\dim_{A,\rho}$ to distinguish between different~$\rho$. Here, as usual, $A$~is a measurable assignment of finite subsets inside a rooted network.  
    
    Due to Proposition~\ref{proposition:Folner_continuous} we can assume that there exists an $r$-F\o lner balanced sequence~$\{F_n\}$ which is continuous for every~$n$, that is, there exists~$r_n$ such that the subset of $F_n(G,o)\subset B_{r_n}(G,o)$ depends only on the~$B_{r_n}(G,o)$. 
    As we assumed that the marks of the network are given from a discrete space, we have only countably many distinct values of~$B_{r_n}$ and, hence, for every~$N$
    \begin{equation}\label{eq_appr0}
        \lim_{n\to \infty}\int\frac{ \dim E^*_{\lambda, F_N}(G,o)}{|F_N(G,o)|} d \rho_n(G,o) = \\
        \int\frac{ \dim E^*_{\lambda, F_N}(G,o)}{|F_N(G,o)|} d \rho(G,o)
    \end{equation}
    Recall that $E^*_{\lambda, F_N}(G,o)$ is the space of~$\lambda$-eigenfunctions with finite support in~$F_N(G, o)$. As $E^*_{\lambda, F_N}(G,o)$ is a subspace of ~$E_\lambda(G)$ we have 
    \begin{multline}\label{eq948675}
        \int\frac{ \dim E^*_{\lambda, F_N}(G,o)}{|F_N(G,o)|} d \rho_n(G,o) \le 
        \int\frac{ \sum_{x \in F_N(G,o)} \langle P_{E_\lambda(G)}\mathbb I_x, \mathbb I_x \rangle}{|F_N(G,o)|} d \rho_n(G,o) 
        = \\ \dim_{F_N, \rho_n} E_\lambda = 
        \int \langle P_{E_\lambda}(G)\mathbb I_o, \mathbb I_o \rangle g_{F_N}(G,o)  d\rho_n(G,o).
    \end{multline}
    The value of the second factor~$g_{F_N}(G,o)$ only depends on the $B_{2r_N}(G,o)$ by formula~\eqref{definition:g_function}, in particular, it attains only countably many values. Hence, for large enough~$n$, the $\rho_n$-distribution of~$g_{F_N}(G,o)$ is close to its~$\rho$ distribution. Let~$\eps$ be a positive constant and $\delta_N = \delta_N(\eps)$ be the value of
    $$
     \int\limits_{g_{F_N} > 1 + \eps} g_{F_N}(G,o) d \rho(G,o).
    $$
    Then
    \begin{multline}\label{eq_467687924}
        \int \langle P_{E_\lambda}(G)\mathbb I_o, \mathbb I_o \rangle g_{F_N}(G,o)  d\rho_n(G,o) \le \\
        \int\limits_{g_{F_N} \le 1 + \eps} \langle P_{E_\lambda}(G)\mathbb I_o, \mathbb I_o \rangle (1 + \eps)  d\rho_n(G,o)  + 
        \int\limits_{g_{F_N} > 1 + \eps} 1 \cdot g_{F_N}(G,o)  d\rho_n(G,o) = \\
        \Tr_{\rho_n}(P_{E_\lambda}) (1 + \eps) + 
        \int\limits_{g_{F_N} > 1 + \eps} (g_{F_N}(G,o) - 1- \eps) d\rho_n(G,o)
    \end{multline}
    We claim that for sufficiently large~$n$
    \begin{equation}\label{eq_455725}
        \int\limits_{g_{F_N} > 1 + \eps} (g_{F_N}(G,o) - 1- \eps ) d\rho_n(G,o)
        \le  2\delta_N.
    \end{equation}
    Although~$g_{F_N}$ is continuous on~$\Gr$, it does not need to be bounded, and the claim is not automatic. We resolve this by taking the function 
    $$
    \bar g_{F_N}(G,o) = \min \{ g_{F_N}(G,o), 1 + \eps\}
    $$
    which is bounded and still continuous. Recall that $\int g_{F_N} d\rho_n = \int g_{F_N} d\rho = 1$ for all~$n$ and~$N$. Due to the weak convergence of~$\rho_n$ to~$\rho$, we have for~$n$ large enough 
    $$
    \int  \bar g_{F_N} d \rho_n \ge \int  \bar g_{F_N} d\rho - \delta_N \ge \int\limits_{g_{F_N} \le 1 + \eps}  g_{F_N} d\rho - \delta_N = 1 - 2\delta_N.
    $$
    Therefore,
    $$
    \int\limits_{g_{F_N} > 1 + \eps} (g_{F_N}(G,o) - 1- \eps ) d\rho_n(G,o) = 1 -  \int  \bar g_{F_N} d \rho_n \le 2 \delta_N.
    $$
    Combining inequalities~\eqref{eq948675}, \eqref{eq_467687924}, and~\eqref{eq_455725} we obtain for every~$N$
    \begin{multline}
        \liminf \dim_{\rho_n}(E_\lambda)  = \liminf \Tr_{\rho_n}(P_{E_\lambda}) \ge \\
        -\frac{2\delta_N}{1 + \eps} + \frac{1}{1+\eps}\int\frac{ \dim E^*_{\lambda, F_N}(G,o)}{|F_N(G,o)|} d \rho(G,o)
    \end{multline}
    Now, when $N$ tends to infinity $\delta_N= \delta_N(\eps)$ goes to zero due to the {F\o lner} balancedness of~$\rho$. As~$\eps$ is arbitrary, we conclude that 
    $$
    \liminf \dim_{\rho_n}(E_\lambda) \ge \dim_\rho(E_\lambda).
    $$
    
    We are left to prove the inequality in the opposite direction. Note, that if~$D$ is a self-adjoint operator then the conclusion follows automatically from the weak convergence of expected spectral measures. To treat the general case we will use the arguments similar to the ones above. Let $\tau_N$ be the following value 
    $$
    \tau_N =\int \frac{|\partial_r F_N(G,o)|}{|F_N(G,o)|} d \rho (G,o).
    $$
    So $\lim_{n \to \infty} \tau_n = 0$ due to $r$-F\o lner property. And for all sufficiently large~$n$, $\int {|\partial_r F_N|}/{|F_N|} d \rho_n < 2 \tau_N$ due to continuity of~$F_N$. Then we have from inequality~\eqref{eq_proofthm1_1}
    \begin{multline}\label{eq_appr1}
        \dim_{F_N,\rho_n} E_\lambda \le 
        \int \frac{\dim E_{\lambda, F_N}(G,o)}{|F_N(G,o)|} d \rho_n \le \\
        \int \frac{\dim E^*_{\lambda, F_N}(G,o)}{|F_N(G,o)|} d \rho_n + \int \frac{|\partial_r F_N|}{|F_N|} d \rho_n \le \int \frac{\dim E^*_{\lambda, F_N}(G,o)}{|F_N(G,o)|} d \rho_n + 2 \tau_N
    \end{multline}
    for sufficiently large~$n$. Then we can analyse 
    $$
    \dim_{F_N, \rho_n} E_\lambda = \int \langle P_{E_\lambda}(G)\mathbb I_o, \mathbb I_o \rangle g_{F_N}(G,o)  d\rho_n(G,o)
    $$
    in a similar manner as before. Let~$\eps$ be a positive constant and let $\tilde \delta_N = \tilde \delta_N (\eps)$ be the measure~$\rho$ of networks~$(G,o)$ such that $g_{F_N}(G,o) \le 1 - \eps$. 
    \begin{multline}\label{eq_appr2}
        \int \langle P_{E_\lambda}(G)\mathbb I_o, \mathbb I_o \rangle g_{F_N}(G,o)  d\rho_n(G,o) \ge \\
        \Tr_{\rho_n}(P_{E_\lambda}) (1 - \eps) - 
        \rho_n\left(\{ g_{F_N} \le 1 - \eps\}\right) \ge \Tr_{\rho_n}(P_{E_\lambda}) (1 - \eps) - 2\tilde\delta_N 
    \end{multline}
    for all sufficiently large~$n$. Inequalities~\eqref{eq_appr1} and~\eqref{eq_appr2} together give
    \begin{multline}
        \limsup_{n \to \infty} \dim_{\rho_n}(E_\lambda)  = \limsup \Tr_{\rho_n}(P_{E_\lambda}) \le \\
        \frac{2\tau_N + 2 \tilde \delta_N}{1 - \eps} + \frac{1}{1 - \eps} \int \frac{\dim E^*_{\lambda, F_N}(G,o)}{|F_N(G,o)|} d \rho_n.
    \end{multline}
    The constants $\tau_N$ and $\tilde\delta_N(\eps)$ tend to zero in~$N$ and~$\eps$ is arbitrary. Relation~\eqref{eq_appr0} and Theorem~\ref{theorem:balanced} give 
    $$
    \limsup_{n \to \infty} \dim_{\rho_n}(E_\lambda) \le \dim_\rho (E_\lambda).
    $$
    Hence, the limit of $\dim_{\rho_n}(E_\lambda)$ exists and is equal to~$\dim_\rho (E_\lambda)$.
\end{proof}

For self-adjoint finite-range operators with integer coefficients, the pointwise convergence of densities of states in the Benjamini-Schramm limits was shown in~\cite{ATV, bordenave2017spectrum}. This implies that the eigenvalues of such operators on sofic random rooted networks are algebraic numbers. These results should be compared with Corollaries~\ref{corollary:algebraic},~\ref{corollary:sofic1}, and~\ref{corollary:approximation} in the present paper.  Although we deal only with amenable networks, our approach has the advantage that it does not require any condition on the coefficients of the operator~$D$ except the finite range and includes operators that are not necessarily self-adjoint or normal. Moreover, it does not rely on the weak convergence of the densities of states. Contrary to that, the methods of~\cite{ATV, bordenave2017spectrum} are based on refining the weak convergence of empirical spectral measures for a Benjamini-Schramm convergent sequence of finite graphs. For a sofic group~$\Gamma$, it was shown in~\cite{JZ, Thom} that the eigenvalues of any element of the group algebra~$\C\Gamma$ or, more generally, any matrix over~$\C\Gamma$, are algebraic over its coefficients. It is unknown to the author whether the conclusion of Corollaries~\ref{corollary:approximation} and~\ref{corollary:algebraic} holds in the context of general finite-range operators on sofic unimodular random networks.

The last corollary that we address in this section states the existence of a basis consisting if finitely supported eigenfunctions for every eigenspace of a finite-range operator. We say that a function $f \in \mathcal{H}$ has finite support if for almost every~$(G,o)$ the section $f_{G,o}$ has finite support. Any such function can be approximated in $\mathcal{H}$ by functions with finite  $\esssup_{G,o} |\supp f_{G,o}|$, that is, the functions with leaf-wise finite support of uniformly bounded size. 

\begin{corollary}\label{corollary:basis}
    Suppose that $\rho$ is $r$-{F\o lner} balanced and~$D$ is an operator of finite range~$r$. Then there exists an orthogonal basis of $E^D_\lambda \le \Hs$ consisting of functions with finite support. 
\end{corollary}
\begin{proof}
    Let us decompose $E_\lambda = E_\lambda^{f} \oplus E_\lambda^{i}$ where $E_\lambda^f$ is the closure of the subset of $\lambda$-eigenfunctions with finite support and $E_\lambda^{i}$ is the orthogonal complement of~$E_\lambda^f$ in~$E_\lambda$.  
    Following the proof of Theorem~\ref{theorem:main} and Theorem~\ref{theorem:balanced} we obtain 
    $$
    \dim_\rho(E_\lambda^f) = \lim_{n \to \infty}\int \frac{\dim \langle f \in E_\lambda(G) \mid \supp f \subset F_n(G,o)\rangle}{|F_n(G,o)|} d\rho(G,o) = \dim_\rho(E_\lambda),
    $$
    where the last equality holds due to Theorem~\ref{theorem:balanced}. Hence, 
    $$
    \dim_\rho(E_\lambda^i) = \dim_\rho(E_\lambda) - \dim_\rho(E_\lambda^f) = 0,
    $$
    and, therefore, $E_\lambda^i = 0$. Then, applying the standard orthogonalisation procedure to a dense sequence of eigenfunctions of finite support, we obtain the orthogonal basis of $E_\lambda$.
\end{proof}
One of the consequences of Corollary~\ref{corollary:basis} is that for any~$\lambda \in \C$, $\rho$-almost surely there is a basis of~$l^2(V(G))$ that consists of finitely supported eigenfunctions. 

\subsection{Localization in amenable unimodular random networks}\label{section_amen_theorem}

In this section, we apply similar methods to unimodular random rooted networks that are amenable in the sense of Section~\ref{subsection_amen}.

\begin{theorem}\label{theorem:isoperimetric}
    Let~$\rho$ be a unimodular random network and~$D$ be an operator of finite range~$r$ such that $\dim_\rho^D(\lambda) > \imath_r$ for some~$\lambda \in \C$. Then with positive probability, the network has an eigenfunction with finite support.  
\end{theorem}

As an immediate consequence we obtain the following.
\begin{corollary}\label{theorem:amenable}
    Let~$\rho$ be an amenable unimodular random network with finite expected degree and~$D$ be an operator of finite range such that $\dim_\rho^D(\lambda) > 0$ for some~$\lambda \in \C$. Then with positive probability, the network has an eigenfunction with finite support.  
\end{corollary}

\begin{proof}[Proof of Theorem~\ref{theorem:isoperimetric}]
    The argument follows the proof of Theorem~\ref{theorem:main}. Let us find a positive~$\eps$ such that ~$\imath_r <\eps < \dim_\rho (E_\lambda)$ and a percolation process $\rho^\prime$ on~$\rho$ with finite clusters such that 
    $$
    \int \frac{|\partial_r F(G,o)|}{|F(G,o)|} d \rho^\prime(G,o) < \eps,
    $$
    where $F(G,o)$ is the cluster of the vertex~$o$ in~$G$. 
    A nice property of~$F$ is that
    $$
    g_F(G,o) = \sum_{o \in F(G,x)} \frac{1}{|F(G,x)|} = 
    \sum_{x \in F(G,o)} \frac{1}{|F(G,o)|} = 1.
    $$
    Using this, one can apply the proof of Theorem~\ref{theorem:main} and obtain
    \begin{equation*}
        \dim_\rho (E_\lambda) \le p_\lambda + \int\frac{|\partial_r F(G,o)|}{|F(G,o)|} d \rho^\prime(G,o) < p_\lambda + \eps,
    \end{equation*}
    where~$p_\lambda$ is the measure $\rho^\prime$ of those networks that have a $\lambda$-eigenfunction with finite support. Note that the measure~$\rho$ of such networks is the same. Together with the choice of~$\eps$, this implies that this probability $p_\lambda$ is positive. 
\end{proof}

The analogues of Corollaries~\ref{corollary:algebraic} and~\ref{corollary:sofic1} for amenable~$\rho$ follow in a similar manner from Corollary~\ref{theorem:amenable}. The analogue of Theorem~\ref{theorem:balanced} can also be derived from Corollary~\ref{theorem:amenable} although one would have to average over different percolations on~$\rho$ instead of~$\rho$ itself, and take eigenfunctions supported on a random cluster of the origin instead of a deterministic F\o lner set. In practice, coming up with a deterministic F\o lner set for each realisation might be an easier task then constructing a percolation process with F\o lner clusters. Further, one has for amenable~$\rho$ the existence of a basis of finitely supported eigenfunctions for every eigenspace, that is, the analogue of Corollary~\ref{corollary:basis}. The pointwise convergence of spectral measures, that is, Corollary~\ref{corollary:approximation} seems to be more complicated in this setting as one would have to approximate an arbitrary percolation on~$\rho$ by a percolation on~$\rho_n$. Doing so in full generality seems to be out of the scope of this paper.  



\section{Long-range graphs}\label{section:long-range}

In this section, we apply Theorem~\ref{theorem:main} to a family of infinite graphs that we call \emph{long-range} graphs, as they describe interaction systems with locally finite but globally unbounded range. These graphs are in general not homogeneous nor quasi-transitive. The definition is inspired by the example of \emph{$\omega$-periodic} graphs that were introduced in~\cite{BH} and further studied in~\cite{BCDN} as Schreier graphs of a self-similar group acting on a binary tree.

\begin{definition}
    A \emph{long-range} graph is a $4$-regular graph $X = (V, E)$ whose vertex set~$V$ is identified with~$\Z$ and the set of edges consists of the set $\{(v, v + 1)\}_{v \in \Z}$ and the other edges (which we call \emph{suspensions}) are such that every vertex has two suspensions one going to the left and one to the right. The root of~$X$ is set to be $0 \in \Z$.

    Here we will assume that the edges of the form $\{(v, v + 1)\}_{v \in \Z}$ are distinguished as well as their orientation. This allows us to determine by any two vertices~$x$ and~$y$ in a long-range graph their order on the line. Formally, this means that we consider networks, which we will still refer to as "graphs," having this distinction in mind.   
\end{definition}

    Note that this definition has straightforward generalization to higher dimensions by replacing~$\Z$ with~$\Z^d$ and drawing a suspension in every one of the~$2d$ directions in the lattice. The proofs for the higher-dimensional case are identical to the ones presented below for dimension one. 
    
\begin{claim}
    Let~$\rho$ be a unimodular random rooted graph concentrated on long-range graphs. Then~$\rho$ is {F\o lner} balanced and, in particular, leaf-wise amenable. 
\end{claim}
\begin{proof}
    For a vertex~$x$ in a long range graph~$G$ denote by~$l(x)$ and~$r(x)$ the lengths of its left and right suspensions respectively. Unimodularity implies that the distribution of the lengths of the suspensions does not depend on the vertex. Indeed, let us show this for the distribution of~$r(x)$. Let~$k$ and~$i$ be fixed integer numbers and consider a function 
    $$
    f(G, x, y) = \mathbb{I}(y - x = k) \mathbb{I}(r(y) = i).
    $$
    The mass transport principle applied to~$f$ gives 
    \begin{multline}
        \rho(r(k) = i) = \int \mathbb{I}(r(k) = i) d \rho (G, 0) = \int \sum_{y \in G} f(G, 0, y) d \rho (G, 0) = \\
        \int \sum_{y \in G} f(G, y, 0) d \rho (G, 0) = \int \mathbb{I}(r(0) = i) d \rho (G, 0) = \rho(r(0) = i).
    \end{multline}
    Here we use the identification of the vertices of the graph with integer numbers. For each positive integer~$i$, let the numbers $p_i^+$ and $p_i^-$  denote the probability that the right and, respectively, left suspension at~$0$ has length~$i$. 
    
    Consider a sequence of subsets $F_n = \{-n, \ldots,n\} \subset \Z = V(G)$. We will show that it has a subsequence that is almost surely F\o lner. Then the measure~$\rho$ would be clearly {F\o lner} balanced. 
    Given a positive integer~$k$, let $S_{k,n}(G,0)$ be the number of vertices in~$F_n$ that have at least one suspension of length more than~$k$. Then
    $$
    \int S_{k,n}(G,0) d\rho(G,0) \le2 n \cdot \left( \sum_{i > k}p_i^+ + \sum_{i > k} p_i^- \right) = n \cdot \eps(k),
    $$
    where $\eps(k)$ is the value of $2\left( \sum_{i > k}p_i^+ + \sum_{i > k} p_i^- \right)$ which goes to zero in~$k$.
    Therefore, 
    \begin{equation}\label{eq:long-range0}
        \rho\left(S_{k,n}(G,0) > n \cdot \sqrt{\eps(k)}\right) \le \sqrt{\eps(k)}.
    \end{equation}
    For every~$m$, choose~$k_m$ such that $\eps(k_m) < 2^{-2m}$ and a sufficiently large~$n_m$, one can take, for instance,~$n_m = 2^m k_m$. One has the following estimate on the size of the boundary of~$F_{n_m}$
    \begin{equation}\label{eq:long-range1}
        \frac{|\partial F_{n_m}(G,0)|}{|F_{n_m}(G,0)|} \le \frac{2k_m + 2S_{k_m,n_m}(G,0)}{2n_m}. 
    \end{equation}
    Let~$A_m$ be the subset of long-range graphs~$(G,0)$ such that 
    $$
    \frac{|\partial F_{n_m}(G,0)|}{|F_{n_m}(G,0)|} > 4 \cdot 2^{-m}.
    $$
    Due to inequalities~\eqref{eq:long-range1} and~\eqref{eq:long-range0}, the probability of~$A_m$ is bounded from above by~$2^{-m}$ and, hence, almost surely only finitely many of~$A_m$-s are satisfied. Hence, the sequence $F_{n_m}$ is indeed F\o lner almost surely. It follows, in particular, that
    $$
    \lim_{m \to \infty}\int \frac{|\partial F_{n_m}(G,0)|}{|F_{n_m}(G,0)|} d \rho(G,0) = 0.
    $$
\end{proof}

\begin{proposition}\label{proposition:long-range_eigenfunctions}
    The adjacency operator of a long-range graph $X = (V, E)$ does not possess an eigenfunction with finite support. 
\end{proposition}

\begin{proof}
    Assume the contrary that there exists an eigenfunction  $f \colon \Z \to \R$  corresponding to an eigenvalue~$\lambda \in \R$. Let $v \in \Z$ be the maximal integer such that~$f(v) \not = 0$. Let~$w \ge v + 1$ be the maximal integer connected to~$v$ by an edge. Then the only vertex in $\supp f$ connected to~$w$ is~$v$. Let $m \in \{1,2\}$ be the number of edges between~$w$ and~$v$. Then
    $$
    \lambda f(w) = m f(v),
    $$
    and, hence, $f(v) = 0$ which contradicts the assumption.
\end{proof}

Note that the same argument applies without changes to any range~$1$ operator with positive $1$-interactions, for example, one can consider edge weights that depend on the length of the edge or add a (random) potential and consider (random) Schrödinger operators.  As a corollary of Theorem~\ref{theorem:main}, we obtain the following result.

\begin{theorem}\label{theorem:long-range}
    Let~$\rho$ be a unimodular random long-range graph. Then the expected spectral measure of the adjacency operator on~$\rho$ is continuous. 
\end{theorem}


We will apply Theorem~\ref{theorem:long-range} to a family of Schreier graphs of an automata group~$\Gamma$ that were studied in~\cite{BCDN}. These graphs first appear in~\cite{BH} where they are called \emph{$\omega$-periodic}. The group~$\Gamma$ is generated by two transformations~$a$ and~$b$ of the space $\{0,1\}^\N$ defined as follows 
\begin{equation}
    a(0w) = 1w;  \quad a(1w) = 0a(w); \quad b(0w) = 0b(w); \quad b(1w) = 1a(w).
\end{equation}
This representation gives rise to a self-similar action of~$\Gamma$ on the rooted infinite binary tree~$T_2$. It is shown in~\cite{BCDN} that the Schreier graphs $G_\omega, \omega\in \partial T_2$ of the action of~$\Gamma$ on the boundary of the tree are a family of long-range graphs. Hence, we can apply Theorem~\ref{theorem:long-range} and conclude that the density of states of the family $\{G_\omega\}_{\omega\in  \partial T_2}$ is continuous. 

One can view more general families of long-range graphs as Schreier graphs of p.m.p. actions of a group. Let~$a$ be an aperiodic invertible measure-preserving transformation of the standard Lebesgue space~$(X,m)$  and~$\xi = \{A_1, A_2,\ldots \}$ be a countable measurable partition. Let us define a transformation~$b$ on~$A_k, \ k =1, 2,\ldots$, as the first return map for the transformation~$a$. This makes~$b$ an invertible measure-preserving transformation of~$(X, m)$. 
\begin{proposition}
    Let~$\Gamma$ be the group generated by the transformations~$a$ and~$b$, and~$S = \{a, b, a^{-1}, b^{-1}\}$ be the set of generators and their inverses. Then $m$-almost any $x \in X$, the Schreier graph $Sch(\Gamma, S,   Stab_\Gamma(x))$ is a long-range graph.
\end{proposition}
\begin{proof}
    By definition, we have $b(x) = a^{n(x)}$, where $n(x)$ is the smallest positive integer~$k$ such that $a^k(x)$ belong to the same cell of~$\xi$ as~$x$. Therefore, the vertex set of the Schreier graph of~$x$ is $\{a^kx\}_{k\in\Z}$ and can be identified with~$\Z$ almost surely, since the action of~$a$ is essentially free. The arcs coming from the action of~$a$ correspond to the edges $(v, v+1)$, and the arcs coming from the action of~$b$ correspond exactly to the suspension arcs in the graph. The fact that~$n(x)$ is positive ensures that each vertex has exactly one suspension to the left of it and exactly one to the right, that is, the Schreier graph $Sch(\Gamma, S,   Stab_\Gamma(x))$ is almost surely a long-range graph. 
\end{proof}

In particular, a random choice of a point~$x \in X$ gives rise to a unimodular long-range graph. Together with Theorem~\ref{theorem:long-range}, this implies the continuity of the expected spectral measure of the family of Schreier graphs~$Sch(\Gamma, S,   Stab_\Gamma(x))$. One can see that the family of $\omega$-periodic graphs can be obtained in this way by taking the transformation~$a$ to be the binary odometer and the sets~$A_k = [2^{-k}, 2^{-k + 1})$ for every positive integer~$k$. In~\cite{BH, BCDN} it was shown that the $\omega$-periodic graphs have intermediate growth. It would be interesting to investigate the growth and other geometric properties of the long-range graphs in general.

\section{Non-balanced unimodular random graphs}

In this section, we give an example of a random unimodular graph that is leaf-wise amenable and not weakly {F\o lner} balanced. Moreover, we will show that this random graph has an atom in the density of states at a point~$\lambda = -3$ for the adjacency operator, and no finitely supported eigenfunctions for~$\lambda$. This example is based on the construction by Kaimanovich (see~\cite{Ka98}) of a leaf-wise amenable random rooted graph which is not amenable as a unimodular measure.

Let $\Gamma = \langle a,b \mid a^2 = b^3 = id\rangle \cong SL(2,\Z)$ be the modular group and $\Gamma \curvearrowright (X,\mu)$ be an ergodic essentially free action of~$\Gamma$ on a standard probability measure space~$(X,\mu)$, e.g., a Bernoulli action. Let $m\colon X \to \mathbb{N}$ be a measurable summable function such that~$m \ge 2$ almost everywhere. Consider the following measure space 
$$
(Y, \rho) = \{(x, k) \in X \times \mathbb{N} \mid k= -m(x),\ldots,-1,1, \ldots, m(x)\},
$$
where the measure~$\rho$ is defined on every subset $m^{-1}(k) \times\{l\} \subset Y$, $0<|l| \le k$  as a copy of~$\mu$ restricted to $m^{-1}(k)$ and normalised by $2\int md\mu$ to be a probability measure. We will identify two isomorphic copies~$X^+$ and~$X^-$ of~$X$ in~$Y$, namely $X \times \{+1\}$ and $X \times \{-1\}$. Let~$\pi$ be the natural projection $\pi \colon Y \to X$ that maps $(x,k)$ to~$x$. We connect two points~$(x,k)$ and~$(y, l)$ in~$Y$ by an edge in two cases. Fist case is when the levels~$k$ and~$l$ are the same and equal to~$1$ or~$-1$ and $y = xs$, $s \in \{a^{\pm}, b^{\pm} \}$. This gives two copies of the Cayley graph $G^0 =Cay(\Gamma, \{a^\pm, b^\pm\})$. As $a = a^{-1}$ the Cayley graph has double edges corresponding to~$a$. The second case is when $x = y$. Then we connect both~$(x,2)$ and~$(x,-2)$ to both~$(x,1)$ and~$(x, -1)$; for $2\le |k| < |l| \le m(x)$ we connect connect $(x, k)$ and $(x, l)$ if $|k-l| = 1$. This gives for every fixed~$x$ a graph consisting of two segments of length~$m(x) - 2$ glued to the opposite vertices of a square. The connected component~$(G,o)$ of a point~$o$ taken randomly in~$Y$ with respect to the measure~$\rho$ gives rise to a unimodular random rooted graph. We will denote it by the same letter~$\rho$. 

\begin{proposition}
    Random rooted graph~$\rho$ is almost surely amenable and has a point mass at~$-3$ in its density of states of the adjacency operator. At the same time~$\rho$ has no finitely supported $-3$ eigenfunctions almost surely. 
\end{proposition}

\begin{proof}
Due to the ergodicity of the action of $\Gamma$ on $(X, \mu)$,  for almost every~$x \in X$, the orbit $\Gamma x$ intersects all non-empty preimages of~$m$. This means that our random rooted graph almost surely has an induced subgraph isomorphic to an interval of arbitrary large length. Hence, we have the leaf-wise amenability. 

Let us show that~$\rho$ almost surely $(G,o)$ has a $-3$-eigenfunction. As the action of~$\Gamma$ on $(X,\mu)$ is essentially free, $(G,o)$ contains $\rho$-almost surely two copies~$G^+$ and~$G^-$ of~$G^0$ located in such a way that there exist two graph isomorphisms~$i^+ \colon G^0 \to G^+$ and $i^-\colon G^0 \to G^-$ such that for every $x\in V(G^0)$ the images~$x^+ =i^+(x)$ and~$x^- = i^-(x)$ are at distance~$2$ of each other in~$G$ and have exactly~$2$ common neighbours. As shown in~\cite{KK25}, $G^0$~has a~$-3$-eigenfunction~$f$. Let us define the function~$h$ on~$G$ as follows. For~$x \in G$, we we put $h(x^+)= f(x)$, $h(x^-)= -f(x)$, and $h(y) =0$ for every vertex $y \in G\setminus (G^+ \cup G^-)$. One can check that~$h$ is indeed a $-3$-eigenfunction by verifying the local condition at the points neighbouring~$ G^+$ and~$G^-$. Therefore, the~$-3$ eigenspace is almost surely non-empty and, hence, the density of states has an atom at~$-3$.

At the same time, one can see that $\rho$-almost surely the rooted graph~$(G,o)$ can not have a $-3$-eigenfunction with finite support. Indeed, suppose there is such a function~$h$. Let~$x_0 \in G^0$ be the furthest point away from the origin among the projections~$\pi(y)$ for $y \in \supp h \subset (G,o)$. Consider the set $\pi^{-1}(x_0)$. By construction, it consists of two segments of equal length glued to opposite vertices of a square. Let us denote $x_0^+$ and $x_0^-$ the vertices of the square to which nothing is glued (these are exactly the ones that belong to~$G^+$ and~$G^-$). Suppose that~$h(x_0^+)\not= 0$. In~$G^0$, there exists a vertex~$x_1$ such that it is further away from the origin than~$x_0$, it is connected to~$x_0$, and all the other neighbours of~$x_1$ are further away from the origin than~$x_0$. Then the point~$x_1^+$ is connected to the only point in $\supp h$ -- the point $x_0^+$. Therefore,~$A_G h (x_1^+) \not = 0$ and~$h$ can not be an eigenfunction. This argument shows that $h(x_0^+) = h(x_0^-) = 0$. Therefore,~$h$ restricted to one of the intervals of~$\pi^{-1}(x_0)$ is non-zero and is a $-3$-eigenfunction. The eigenvalues of an interval do not exceed~$2$ in absolute value, thus we obtain a contradiction and~$(G,o)$ does not have any finitely supported $-3$-eigenfunctions almost surely. In particular, $\rho$ is not weakly {F\o lner} balanced due to Theorem~\ref{theorem:main}.
\end{proof}

\bibliographystyle{plain} 
\bibliography{refs} 

\begin{center}

\textsc{Georgii Veprev, Section de math\'ematiques, Universit\'e de Gen\`eve, 1205~Gen\`eve, Switzerland}\\
\textit{E-mail address: }\texttt{georgii.veprev@gmail.com}\\[2ex]
\end{center}

\end{document}